\documentclass[reqno,12pt] {amsart}
\usepackage{amsmath}
\usepackage[usenames, dvipsnames]{color}

\newtheorem{theorem}{Theorem}[section]
\newtheorem{lemma}[theorem]{Lemma}
\newtheorem{prop}[theorem]{Proposition}

\theoremstyle{definition}

\theoremstyle{remark}
\newtheorem{remark}[theorem]{Remark}

\newcommand{\mysection}[1]{\section{#1}
\setcounter{equation}{0}}

\newcommand{\bR}{\mathbb R}

\DeclareMathOperator{\esssup}{ess\,sup}

\renewcommand{\epsilon}{\varepsilon}

\begin{document}
\title[The Navier-Stokes equations] {The Navier-Stokes equations in the critical Lebesgue space}

\author[H. Dong]{Hongjie Dong}
\address[H. Dong]{Division of Applied Mathematics, Brown University, 182 George Street, Box F, Providence, RI 02912, USA}
\email{Hongjie\_Dong@brown.edu}
\thanks{Hongjie Dong was partially supported by the National Science Foundation under agreement No. DMS-0111298 and DMS-0800129. Dapeng Du was partially supported by China Postdoctor Science Fund CPSF 20070410683.}

\author[D. Du]{Dapeng Du}
\address[D. Du]{School of Mathematical Sciences,
Fudan University, Shanghai 200433, P.R. China}
\email{dpdu@fudan.edu.cn}

\date{\today}



\begin{abstract}
We study regularity criteria for the $d$-dimensional incompressible Navier-Stokes equations. We prove in this paper that if  $u\in L_\infty^tL_{d}^x((0,T)\times \bR^d)$ is a Leray-Hopf weak solution, then $u$ is smooth and unique in $(0,T)\times \bR^d$. This generalizes a result by Escauriaza, Seregin and \v Sver\'ak \cite{ESS}. Additionally, we show that if $T=\infty$ then $u$ goes to zero as $t$ goes to infinity.
\end{abstract}

\maketitle

\mysection{Introduction}

In this paper we consider the incompressible Navier-Stokes
equations in $d$ spatial dimensions with unit viscosity and
zero external force:
\begin{equation}
                            \label{NSE}
\partial_t u+u\cdot\nabla u-\Delta u+\nabla p=0,\quad \text{div}\, u=0
\end{equation}
for $x\in\bR^d$ and $t\geq 0$ with the initial condition
\begin{equation}
                                \label{NSE2}
u(0,x)=a(x),\quad x\in\bR^d.
\end{equation}
Here $u$ is the velocity and $p$ is the pressure.

For sufficiently regular data $a$, the local
strong solvability of such problems is well known (see, for example,
\cite{kato}, \cite{giga2}, \cite{Taylor} and \cite{tataru}). The solution is unique and locally smooth in both spatial and time variables. On the other hand, the global in time strong solvability is an outstanding open problem for $d\ge 3$.

Another important type of solutions are called {\em Leray-Hopf} weak solutions (see Section \ref{sec2.1} for the notation and definition).
In the pioneering works of Leray \cite{leray} and Hopf \cite{Hopf}, it is shown that for any divergence-free vector field $a\in L_2$, there exists at least one Leray-Hopf weak solution of the Cauchy problem \eqref{NSE}-\eqref{NSE2} on $(0,\infty)\times\bR^d $. Although the problems of uniqueness and regularity of Leray-Hopf weak solutions are still open, since the seminal work of Leray there is an extensive literature on conditional results under various criteria. The most well-known condition is so-called Ladyzhenskaya-Prodi-Serrin condition, that is for some $T>0$
\begin{equation}
                                    \label{eq10.56}
u\in L_r^tL_q^x(\bR^{d+1}_T),
\end{equation}
where the pair $(r,q)$ satisfies
\begin{equation*}
\frac 2 r+\frac d q\le 1,\quad q\in (d,\infty].
\end{equation*}
Under the condition \eqref{eq10.56}, the uniqueness of Leray-Hopf weak solutions was proved by Prodi \cite{Prodi} and Serrin \cite{Serrin62}, and the smoothness was obtained by Ladyzhenskaya \cite{Lady}. For further results, we refer the reader to \cite{giga86}, \cite{Serrin63}, \cite{Struwe} and recent \cite{CS}, and references therein. The borderline case $(r,q)=(\infty,d)$ is much more subtle since the result cannot be proved by usual methods using the local smallness of certain norms of $u$ which are invariant under the natural scaling
\begin{equation}
					\label{eq11.35}
u(t,x)\to \lambda u(\lambda^2 t,\lambda x),\quad p(t,x)\to \lambda^2 p(\lambda^2 t,\lambda x).
\end{equation}
For $d=3$, this case was studied recently by Escauriaza, Seregin and \v Sver\'ak in a remarkable paper \cite{ESS}. The main result of \cite{ESS} is the following theorem.

\begin{theorem}[Escauriaza, Seregin and \v Sver\'ak]
                                        \label{thm1.1}
Let $d=3$. Suppose that $u$ is a Leray-Hopf weak solution of the Cauchy problem \eqref{NSE}-\eqref{NSE2} in $(0,T)\times \bR^3$ and $u$ satisfies the condition \eqref{eq10.56} with $(r,q)=(\infty,3)$. Then $u\in L_5((0,T)\times \bR^3)$, and hence it is smooth and unique in $(0,T)\times \bR^3$.
\end{theorem}

Before we give a description of Theorem \ref{thm1.1}, we shall recall another important concept, the partial regularity of weak solutions. The study of partial regularity of the Navier-Stokes equations was originated by Scheffer in a series of papers \cite{Sch1, Sch2, Sch4}. In three space dimensions, he established various partial regularity results for weak solutions satisfying the so-called local energy inequality.
For $d=3$, the notion of {\em suitable weak solutions} was introduced in a celebrated paper \cite{CKN} by Caffarelli,
Kohn and Nirenberg. They called a pair $(u,p)$ a suitable weak solution if $u$ has finite energy
norm, $p$ belongs to the Lebesgue space $L_{5/4}$, $u$ and $p$ are weak solutions to the Navier-Stokes equations and satisfy a local energy inequality. It is proved
that, for any suitable weak solution $(u,p)$, there is an open subset in which the velocity field $u$ is H\"older continuous, and the complement of it has zero 1-D Hausdorff measure. In \cite{flin}, with zero external force, Lin gave a more direct and sketched proof of Caffarelli, Kohn and Nirenberg's result. A detailed treatment was then later given by Ladyzhenskaya and Seregin in \cite{OL2}. For other results in this direction, we refer the reader to \cite{Struwe}, \cite{tsai}, \cite{DongDu07} and references therein.

The proofs in \cite{ESS} are highly nontrivial and rely on certain regularity criteria in the light of \cite{CKN}, \cite{flin} and \cite{OL2}. That is, roughly speaking, if some scaling invariant quantities are small then the solution is locally regular. Another main ingredient of the proof is a backward uniqueness theorem of heat equations with bounded coefficients of lower order terms in the half space (see also \cite{ESS2}). Under an additional assumption on the pressure, there are some extensions of Theorem \ref{thm1.1} to the half space case and the bounded domain case; we refer the reader to \cite{Seregin05} and \cite{MiShi} for some results in this direction. Another interesting open problem is the extension to the higher dimensional Navier-Stokes equations. It seems to us that the argument in \cite{ESS} breaks down in several places when $d\ge 4$. In particular, the regularity criterion, Theorem 2.2 \cite{ESS}, is unknown for the higher dimensional Navier-Stokes equations.

We now state the main results of the article.

\begin{theorem}
                                        \label{thm1}
Let $d\ge 3$, $K>0$ and $T\in (0,\infty)$. Suppose that $u$ is a Leray-Hopf weak solution of the Cauchy problem \eqref{NSE}-\eqref{NSE2} in $(0,T)\times \bR^d$ and $u$ satisfies the condition
\begin{equation}
                                    \label{eq11.11}
u\in L_\infty^tL_d^x((0,T)\times \bR^d),\quad \|u\|_{L_\infty^tL_d^x((0,T)\times \bR^d)}\le K.
\end{equation}
Then $u\in L_{d+2}((0,T)\times \bR^d)$, and hence it is smooth and unique in $(0,T)\times \bR^d$.
\end{theorem}

\begin{theorem}
                                        \label{thm2}
Let $d\ge 3$ and $K>0$. Suppose that $u$ is a Leray-Hopf weak solution of the Cauchy problem \eqref{NSE}-\eqref{NSE2} in $(0,\infty)\times \bR^d$ and $u$ satisfies the condition
\begin{equation}
                                    \label{eq11.11b}
u\in L_\infty^tL_d^x((0,\infty)\times \bR^d),\quad \|u\|_{L_\infty^tL_d^x((0,\infty)\times \bR^d)}\le K.
\end{equation}
Then $u$ is smooth and unique in $(0,\infty)\times \bR^d$. Moreover, we have
\begin{equation}
                            \label{eq10.22}
\lim_{t\to\infty}\|u(t,\cdot)\|_{L_\infty}=0.
\end{equation}
\end{theorem}

We give a brief description of our argument. As in \cite{ESS} we prove by contradiction and  blow up the solution near a singular point at the first blow-up time to obtain a sequence of solutions $\{u_k\}$.
The limiting function $u_\infty$ of this sequence is a suitable weak solution of the Navier-Stokes equations.
Note that the solutions $u_k$ are smooth before the first blow-up time.
As  we mentioned before, we are not able to establish a regularity criterion similar to Theorem 2.2 \cite{ESS}, which says if certain scaling invariant quantities are small then the solution is locally H\"older continuous. Instead we  use a modified one.
Roughly speaking, we show that if the solutions are smooth, the $L^t_{\infty}L^x_d$ norm is bounded and some
scaling invariant quantities are small,
then we have a priori $L_{\infty}$ bound for the solutions on a much smaller ball.
Here the point is the a priori $L_{\infty}$ bound only depends on the
$L^t_{\infty}L^x_d$ norm and the dimension.
This regularity criterion together with the $L_p$-convergence of $u_k$ yields the local boundedness of $u_\infty$ outside a large cylinder. The local boundedness implies the local smoothness of $u_\infty$. Then we use the backward uniqueness proved \cite{ESS} to see that $u_\infty$ is equivalent to zero outside a large cylinder, which further implies that $u_\infty\equiv 0$ by using the spatial analyticity of strong solutions and the weak-strong uniqueness of the Navier-Stokes equations.
This  means the sequence $u_k$ converges to zero in $L_p$ on any compact set. Going back to the original solution $u$ we see that the modified regularity criterion applies,
which gives a contradiction and proves Theorem \ref{thm1}. To prove Theorem \ref{thm2}, we notice that $u$ is in $L_4((0,\infty)\times \bR^d)$, which implies the smallness of its $L_4$ norm in $(T,\infty)\times \bR^d$ for large $T$. Then we use the modified regularity criterion again and the scaling \eqref{eq11.35}.

We remark that a decay result similar to that of Theorem \ref{thm2} was obtained in \cite{GIP} by using a completely different method.

The remaining part of the article is organized as follows. We give a few definitions and prove several preliminary results in the next section. In Section \ref{sec3}, we prove a key estimate (Proposition \ref{prop3.1}) about the scaling invariant quantities and construct a sequence of solutions by blowing up the solution at a singular point. Section \ref{sec4} is devoted to the proof of a local boundedness estimate (Theorem \ref{thm4.1}). We finish the proof of Theorem \ref{thm1} and \ref{thm2} in Section \ref{sec5}.

\mysection{Preliminaries}
						\label{sec2}
We make a few preparations in this section. We use the notation in \cite{OL2}. Let $\omega$ be a domain in
some finite-dimensional space. Denote $L_p(\omega ;\bR^n)$ and
$W^k_p(\omega;\bR^n)$ to be the usual Lebesgue and Sobolev spaces of
functions from $\omega$ into $\bR^n$. Denote the norm of the spaces
$L_p(\omega;\bR^n)$ and $W^k_p(\omega;\bR^n)$ by
$\|\cdot\|_{L_p(\omega)}$ and $\|\cdot\|_{W^k_p(\omega)}$
respectively. As usual, for any measurable function $u=u(x,t)$ and
any $p,q\in [1,+\infty]$,  we define
$$
\|u(x,t)\|_{L_t^pL_x^q}:=\big\|\|u(x,t)\|_{L_x^q}\big\|_{L_t^p}.
$$

For summable functions $p,u=(u_i)$ and $\tau=(\tau_{ij})$, we use
the following differential operators
$$
\partial_t u=u_t=\frac{\partial u}{\partial t},\,\,\,\,
u_{,i}=\frac{\partial u}{\partial x_i},\,\,\,\, \nabla
p=(p_{,i}),\,\,\,\,\nabla u=(u_{i,j}),
$$
$$
\text{div}\,u=u_{i,i},\,\,\,\,\text{div}\,\tau=(\tau_{ij,j}),\,\,\,\,
\Delta u=\text{div}\nabla u,
$$
which are understood in the sense of distributions. We use the notation of spheres, balls and parabolic cylinders,
$$
S(x_0,r)=\{x\in \bR^4\,\vert\, |x-x_0|=r\},\quad S(r)=S(0,r),\quad
S=S(1);
$$
$$
B(x_0,r)=\{x\in \bR^4\,\vert\, |x-x_0|<r\},\quad B(r)=B(0,r),\quad B=B(1);
$$
$$
Q(z_0,r)=B(x_0,r)\times (t_0-r^2,t_0),\quad Q(r)=Q(0,r),\quad Q=Q(1).
$$
Also we denote mean values of summable functions as follows
$$
[u]_{x_0,r}(t)=\frac{1}{|B(r)|}\int_{B(x_0,r)}u(x,t)\,dx,
$$
$$
(u)_{z_0,r}=\frac{1}{|Q(r)|}\int_{Q(z_0,r)}u\,dz.
$$

We recall the following well-known interpolation inequality.
\begin{lemma}
                    \label{lemmamulti}
For any functions $u\in W^1_2(\bR^d)$ and any $q\in [2,2d/(d-2)]$
and $r>0$,
$$
\int_{B_r}|u|^q\,dx\leq N(q)\Big[\big(\int_{B_r}|\nabla u|^2\,
dx\big)^{d(q/4-1/2)}
\big(\int_{B_r}|u|^2\,dx\big)^{q/2-d(q/4-1/2)}
$$
\begin{equation}
                        \label{eq11.30}
 +r^{-d(q-2)/2}\big(\int_{B_r}|u|^2\,dx\big)^{q/2}\Big].
\end{equation}
\end{lemma}

\subsection{Leray-Hopf weak solutions}
						\label{sec2.1}
We denote $\dot C_0^\infty$ the space of all divergence-free infinitely differentiable vector fields with compact support in $\bR^d$. Let $\dot J$ and $\dot J^1_2$ be the closure of $\dot C_0^\infty$ in the spaces $L_2$ and $W^1_2$, respectively. For any $T\in (0,\infty]$, denote
$$
\bR^{d+1}_T=(0,T)\times \bR^d.
$$
By a Leray-Hopf weak solution of \eqref{NSE}-\eqref{NSE2} in $\bR^{d+1}_T$, we mean a vector field $u$ 
such that:

i) $u\in L_\infty(0,T;\dot J)\cap L_2(0,T;\dot J^1_2)$;

ii) the function $t\to\int_{\bR^d}u(t,x)\cdot w(x)\,dx$ is continuous on $[0,T]$ for any $w\in L_2$;

iii) the equation \eqref{NSE} holds weakly in the sense that for any $w\in \dot C_0^\infty(\bR^{d+1}_T)$,
\begin{equation}
                                    \label{eq10.11}
\int_{\bR^{d+1}_T}(-u\cdot \partial_t w-u\otimes u\,:\,\nabla w+\nabla u\,:\,\nabla w)\,dx\,dt=0;
\end{equation}

iv) The energy inequality:
\begin{equation*}
\frac 1 2\int_{\bR^d}|u(t,x)|^2\,dx+\int_{\bR^d_{t}}|\nabla u|^2\,dx\,ds\le \frac 1 2\int_{\bR^d}|a(x)|^2\,dx
\end{equation*}
holds for any $t\in [0,T]$, and we have
$$
\|u(t,\cdot)-a(\cdot)\|_{L_2}\to 0\quad \text{as}\,\,t\to 0.
$$
It is well known that for any $a\in \dot J$, there exists at least one Leray-Hopf weak solution of the Cauchy problem \eqref{NSE}-\eqref{NSE2} on $(0,\infty)\times\bR^d $ (see \cite{leray} and \cite{Hopf}).

\subsection{Suitable weak solutions}

The definition of suitable weak solutions was introduced in \cite{CKN} (see also \cite{flin} and \cite{OL2}). Let $\omega$ be an open set in $\bR^d$. By a suitable weak solution of the Navier-Stokes equations on the set $(0,T)\times\omega$, we mean a pair $(u,p)$ such that

i) $u\in L_\infty(0,T;\dot J)\cap L_2(0,T;\dot J^1_2)$ and $p\in L_{d/2}((0,T)\times\omega)$;

ii) $u$ and $p$ satisfy equation \eqref{NSE} in the sense of distributions \eqref{eq10.11}.

iii) For any $t\in (0,T)$ and for any nonnegative function $\psi\in C_0^\infty(\bR^d)$ vanishing in a neighborhood of the parabolic boundary $ \{t=0\}\times\omega\cup [0,T]\times\partial\omega$, we have the local energy inequality
\begin{multline}
\esssup_{0<s\le t}\int_{\omega}|u(s,x)|^2\psi(s,x)\,dx+2\int_{ (0,t)\times \omega}|\nabla u|^2\psi\,dx\,ds \\
                    \label{energy}
\leq \int_{(0,t)\times \omega}\{|u|^2(\psi_t+\Delta \psi)+(|u|^2+2p)u\cdot
\nabla\psi\}\,dx\,ds.
\end{multline}

\subsection{Scaling invariant quantities}
The following notation will be used throughout the article:
$$
A(r)=A(r,z_0)=\esssup_{t_0-r^2\leq t\leq
t_0}\frac{1}{r^{d-2}}\int_{B(x_0,r)}|u(x,t)|^2\,dx,
$$
$$
E(r)=E(r,z_0)=\frac{1}{r^{d-2}}\int_{Q(z_0,r)}|\nabla u|^2\,dz,
$$
$$
C(r)=C(r,z_0)=\frac{1}{r^{d-4/(d+1)}}
\int_{Q(z_0,r)}|u|^{2(d+3)/(d+1)}\,dz,
$$
$$
D(r)=D(r,z_0)=\frac{1}{r^{d-4/(d+1)}}
\int_{Q(z_0,r)}|p|^{(d+3)/(d+1)}\,dz.
$$
We notice that these quantities are all invariant under the natural scaling \eqref{eq11.35}.

We shall use the following two lemmas involving these quantities.

\begin{lemma}
 				\label{lem2.1}
Let $\rho>0$, $\epsilon>0$ be constants and $(u,p)$ a pair of suitable weak solution of \eqref{NSE}. Suppose $Q(z_0,\rho)\subset \bR^{d+1}_T$ and
$$
C(\rho)+D(\rho)\le \epsilon^{2(d+3)/(d+1)}.$$
Then under the condition \eqref{eq11.11}, we have
\begin{equation*}
A(\rho/2)+E(\rho/2)\le N\epsilon^2.
\end{equation*}
\end{lemma}
\begin{proof}
By a scaling argument, we may assume without loss of generality that $\rho=1$. In the
energy inequality \eqref{energy}, we put $t=t_0$ and choose a
suitable smooth cut-off function $\psi$ such that
$$
\psi\equiv 0\,\,\text{in}\,\,\bR^{d+1}_{t_0}\setminus Q(z_0,1), \quad 0\leq
\psi\leq 1\,\,\text{in}\,\,\bR^{d+1}_T,
$$
$$
\psi\equiv 1 \,\,\text{in}\,Q(z_0,1/2),\quad |\nabla
\psi|<N,\,\, |\partial_t \psi|+|\nabla^2
\psi|<N\,\,\text{in}\,\,\bR^{d+1}_{t_0}.
$$
By using \eqref{energy}, we get
\begin{align*}
A(1/2)+2E(1/2)\leq N\int_{Q(z_0,1)}|u|^2\,dz+N\int_{Q(z_0,1)}(|u|^2+2|p|)|u|\,dz.
\end{align*}
Due to H\"older's inequality, one can obtain
$$
\int_{Q(z_0,1)}|u|^2\,dz\leq N(C(1))^{(d+1)/(d+3)}\le N\epsilon^2,
$$
and
\begin{align*}
&\int_{Q(z_0,1)}(|u|^2+2|p|)|u|\,dz\\
&\le N
\left(\int_{Q(z_0,1)}|u|^{\frac{d+3}{2}}\,dz\right)^{\frac{2}{d+3}}
\left(\int_{Q(z_0,1)}|u|^{\frac{2(d+3)}{d+1}}+|p|^{\frac{d+3}{d+1}}\,dz\right)^{\frac{d+1}{d+3}}\\
&\le N\left(\int_{Q(z_0,1)}|u|^{d}\,dz\right)^{\frac{1}{d}}(C(1)+D(1))^{(d+1)/(d+3)}\\
&\le N\epsilon^{2},
\end{align*}
where in the last inequality we used \eqref{eq11.11}.
The conclusion of Lemma \ref{lem2.1} follows immediately.
\end{proof}

\begin{lemma}
 			\label{lem2.3}
Suppose $\gamma\in (0,1/2]$, $\rho>0$ are constants and $Q(z_0,\rho)\in \bR^{d+1}_T$. Then we have
\begin{equation}
                                \label{eq11.29}
D(\gamma\rho)\leq N\big[\gamma^{-d+4/(d+1)}C(\rho)
+\gamma^{4/(d+1)} D(\rho)\big].
\end{equation}
\end{lemma}
\begin{proof}
Let $\eta(x)$ be a smooth function on $\bR^d$ supported in the unit
ball $B(1)$, $0\leq \eta\leq 1$ and $\eta\equiv 1$ on $\bar B(2/3)$.
It is known that for a.e.
$t\in (t_0-\rho^2,t_0)$, in the sense of distribution one has
\begin{equation}
				\label{pressure}
\Delta p=D_{ij}\big(u_iu_j\big).
\end{equation}
For these $t$, we consider the decomposition
$$
p=p_{x_0,\rho}+h_{x_0,\rho}\quad \text{in}\,\, B(x_0,\rho),
$$
where $p_{x_0,\rho}$ is the Newtonian potential of
$$D_{ij}\big(u_iu_j\big)\eta((x-x_0)/\rho).$$ Then $h_{x_0,\rho}$
is harmonic in $B(x_0,2\rho/3)$.

Denote $r=\gamma \rho$. By using the Calder\'on-Zygmund estimate, one has
\begin{align}
&\int_{Q(z_0,r)}|p_{x_0,\rho}(x,t)|^{(d+3)/(d+1)}\,dz\nonumber\\
&\le \int_{Q(z_0,\rho)}|p_{x_0,\rho}(x,t)|^{(d+3)/(d+1)}\,dz\nonumber\\
			\label{eq2.19}
&\le \int_{Q(z_0,\rho)}|u|^{2(d+3)/(d+1)}\,dz.
\end{align}
Since $h_{x_0,\rho}$ is harmonic in $B(x_0,2\rho/3)$, any Sobolev
norm of $h_{x_0,\rho}$ in a smaller ball can be estimated by any of
its $L_p$ norm in $B(x_0,2\rho/3)$. Thus, one obtains
\begin{align}
&\int_{B(x_0,r)}|h_{x_0,\rho}|^{(d+3)/(d+1)}\,dx\nonumber\\
&\leq Nr^{d}\sup_{B(x_0,r)}
|h_{x_0,\rho}|^{(d+3)/(d+1)}\,dx \nonumber\\
					\label{eq2.28}
&\leq
Nr^{d}\rho^{-d}\int_{B(x_0,\rho)}|h_{x_0,\rho}|^{(d+3)/(d+1)}\,dx.
\end{align}
Integrating \eqref{eq2.28} in $t\in (t_0-r^2,t_0)$, we obtain
\begin{align}
&\int_{Q(z_0,r)}|h_{x_0,\rho}|^{(d+3)/(d+1)}\,dz\nonumber\\
&\le Nr^{d}\rho^{-d}\int_{Q(z_0,\rho)}|h_{x_0,\rho}|^{(d+3)/(d+1)}\,dz\nonumber\\
&\le Nr^{d}\rho^{-d}\int_{Q(z_0,\rho)}|p|^{(d+3)/(d+1)}+|p_{x_0,\rho}|^{(d+3)/(d+1)}\,dz\nonumber\\
 			\label{eq2.32}
&\le Nr^{d}\rho^{-d}\int_{Q(z_0,\rho)}|p|^{(d+3)/(d+1)}\,dz+
N\int_{Q(z_0,\rho)}|u|^{2(d+3)/(d+1)}\,dz,
\end{align}
where we used \eqref{eq2.19} in the last inequality.
By combining \eqref{eq2.19} and \eqref{eq2.32} we reach \eqref{eq11.29}. The lemma is proved.
\end{proof}

\subsection{Strong solutions and spatial analyticity}

We recall the following local strong solvability of \eqref{NSE}-\eqref{NSE2} (see, for example,
\cite{kato}, \cite{giga2}, \cite{Taylor} and \cite{tataru}), and the spatial analyticity of strong solutions (see, for example, \cite{GigaSawada} and \cite{DongLi}).

\begin{prop}
                                    \label{prop3.53}
For any divergence-free initial data $a\in L_p(\bR^d),p\ge d$, the Cauchy problem \eqref{NSE}-\eqref{NSE2} has a unique strong solution $u\in C([0,\delta);L_p(\bR^d))$ for some $\delta>0$. Moreover,  $u$ is infinitely differentiable and spatial analytic for $t\in (0,\delta)$.
\end{prop}

\mysection{A blowup procedure}
				\label{sec3}

We begin this section by proving the following key estimate, which shows if the quantities $C$ and $D$ are sufficiently small in a cylinder, then they must be also small in any sub-cylinder.

\begin{prop}
 				\label{prop3.1}
Let $(u,p)$ be a pair of suitable weak solution of \eqref{NSE}. Suppose that $Q(z_0,\rho)\subset \bR^{d+1}_T$ and the condition \eqref{eq11.11} holds. Then for any $\epsilon_0>0$ there exists an $\epsilon^*>0$ depending only on $\epsilon_0$ and $d$ such that if
$$
C(\rho,z_0)+D(\rho,z_0)\le \epsilon^*,$$
then we have
\begin{equation*}
C(r,z_1)+D(r,z_1)\le \epsilon_0
\end{equation*}
for any $z_1\in Q(z_0,\rho/2)$ and $r\in (0,\rho/2)$.
\end{prop}

Proposition \ref{prop3.1} follows immediately from the next lemma by using a covering argument and an iteration.

\begin{lemma}
					\label{lem3.2}
Let $(u,p)$ be a pair of suitable weak solution of \eqref{NSE}. Suppose that $Q(z_0,\rho)\subset \bR^{d+1}_T$ and the condition \eqref{eq11.11} holds. Then there exist universal constants $\epsilon^*>0$ and $\gamma\in (0,1/4]$ such that for any $\epsilon\in (0,\epsilon^*]$ if
$$
C(\rho,z_0)+D(\rho,z_0)\le \epsilon,
$$
then we have
\begin{equation*}
C(\gamma\rho,z_0)+D(\gamma\rho,z_0)\le \epsilon.
\end{equation*}
\end{lemma}
\begin{proof}
As before, one may assume $\rho=1$. We prove by contradiction. Let $\gamma\in (0,1/4]$ be a constant to be specified later. Suppose there exist a decreasing sequence $\{\epsilon_k\}$ converging to $0$, and a sequence of pairs of suitable weak solutions $(u_k,p_k)$ such that
\begin{align}
 					\label{eq3.28}
C(1,z_0,u_k,p_k)+D(1,z_0,u_k,p_k)&\le \epsilon_k^{2(d+3)/(d+1)},\\
					\label{eq3.29}
C(\gamma,z_0,u_k,p_k)+D(\gamma,z_0,u_k,p_k)&> \epsilon_k^{2(d+3)/(d+1)}.
\end{align}
By Lemma \ref{lem2.1}, one also has
\begin{equation}
 					\label{eq3.31}
A(1/2,z_0,u_k,p_k)+B(1/2,z_0,u_k,p_k)\le N \epsilon_k^{2},
\end{equation}
where the constant $N$ is independent of $k$.

We define $(v_k,q_k)=(u_k/\epsilon_k,q_k/\epsilon_k)$. Then $(v_k,q_k)$ is a suitable weak solution of
\begin{equation}
				\label{eq4.08}
\partial_t v_k+\epsilon_k v_k\cdot\nabla v_k-\Delta v_k+\nabla q_k=0,\quad \text{div}\, v_k=0.
\end{equation}

From \eqref{eq3.28},
\eqref{eq3.29} and \eqref{eq3.31}, we get
\begin{align}
  					\label{eq3.33}
 C(1,z_0,v_k,q_k)+D(1,z_0,v_k,q_k)&\le 1,\\
					\label{eq3.34}
C(\gamma,z_0,v_k,q_k)+D(\gamma,z_0,v_k,q_k)&> 1,\\
					\label{eq12.05}
A(1/2,z_0,v_k,q_k)+B(1/2,z_0,v_k,q_k)&\le N.
\end{align}
By using \eqref{eq12.05}, applying the interpolation inequality \eqref{eq11.30} with $q=2(d+2)/d$ and integrating in $t$, we bound $\|v_k\|_{L_{2(d+2)/d}(Q(z_0,1/2))}$ by $N$. Thus by the H\"older's inequality,
$$
\|v_k\cdot\nabla v_k\|_{L_{(d+2)/(d+1)}(Q(z_0,1/2))}\le N
$$
Due to the coercive estimate for the Stokes system (see, for instance, \cite{MareSolo}) with a suitable cut-off function, we reach
\begin{equation*}
\int_{Q(z_0,1/3)}\left(|v_k|^{\frac{2(d+2)} d}+|\partial_t v_k|^{\frac {d+2} {d+1}}+|D^2 v_k|^{\frac {d+2} {d+1}}+|\nabla q_k|^{\frac {d+2} {d+1}}\right)\,dz\le N,
\end{equation*}
where the constant $N$ is independent of $k$. Thanks the compact embedding theorem and \eqref{eq3.33}, there exist
$$
v\in L_{2(d+3)/(d+1)}(Q(z_0,1/3)),\quad
q\in  L_{(d+3)/(d+1)}(Q(z_0,1/3)),
$$
and a subsequence, which is still denoted by $(v_k,q_k)$ such that
\begin{align}
                        \label{eq8.44}
v_k&\to v\,\, \text{in}\,\, L_{2(d+3)/(d+1)}(Q(z_0,1/3)),\\                        q_k&\rightharpoonup q\,\, \text{in}\,\, L_{(d+3)/(d+1)}(Q(z_0,1/3)).\nonumber
\end{align}
This together with \eqref{eq4.08} implies
\begin{equation}
                            \label{eq4.11}
\partial_t v-\Delta v+\nabla q=0,\quad \text{div}\, v=0.
\end{equation}
Moreover,
$$
\|v\|_{L_{2(d+3)/(d+1)}(Q(z_0,1/3))}
+\|q\|_{L_{(d+3)/(d+1)}(Q(z_0,1/3))}\le N.
$$
By the classical estimate of the Stokes system, one has
$$
\sup_{Q(z_0,1/4)}|v|\le N,
$$
which gives
$$
C(\gamma,z_0,v,q)\le N\gamma^{4/(d+1)}.
$$
This contradicts \eqref{eq3.34} and \eqref{eq8.44}, if we choose $\gamma$ sufficiently small. The lemma is proved.
\end{proof}

\begin{lemma}
				\label{lem3.3}
Under the assumptions of Theorem \ref{thm1}, we have
\begin{equation}
				\label{eq5.25}
\|u(t,\cdot)\|_{L_d(\bR^d)}\le N,
\end{equation}
for each $t\in [0,T]$,
and
\begin{equation}
				\label{eq5.28}
u\in L_4(\bR^{d+1}_T),\quad \partial_t u,D^2u,\nabla p\in L_{4/3}((\delta,T)\times \bR^d),
\end{equation}
for any $\delta\in (0,T)$. Moreover, $(u,p)$ is a pair of suitable weak solution of \eqref{NSE} in $\bR^{d+1}_T$.
\end{lemma}
\begin{proof}
The first assertion is due to \eqref{eq11.11} and the weak continuity of Leray-Hopf weak solutions. By using Lemma \ref{lemmamulti} with $q=2d/(d-2)$ and $r=\infty$, we have
\begin{equation*}
\|u\|_{L_2^tL_{2d/(d-2)}^x(\bR^{d+1}_T)}\le N,
\end{equation*}
which together with \eqref{eq5.25} and the H\"older's inequality yields
$$
\|u\|_{L_4(\bR^{d+1}_T)}\le N,\quad \|u\cdot \nabla u\|_{L_{4/3}(\bR^{d+1}_T)}\le N.
$$
Thus, \eqref{eq5.28} follows from the coercive estimate for the Stokes system. Finally, due to the pressure equation \eqref{pressure} and the Calder\'on-Zygmund estimate, $p\in L_\infty^tL_{d/2}^x(\bR^{d+1}_T)$. Therefore, it is clear that $(u,p)$ is a suitable weak solution. The lemma is proved.
\end{proof}

\begin{remark}
                        \label{rem3.3}
From \eqref{eq5.28}, one can infer that $u\in C((0,T];L_{4/3}(B_R))$ for any $R>0$. This combined with \eqref{eq5.25} and the H\"older's inequality, we get $u\in C((0,T];L_{p}(B_R))$ for any $p\in [1,d)$.
\end{remark}

Because of the local strong solvability and the weak-strong uniqueness (see, for instance, \cite{Wahl}), we know that $u$ is regular for $t\in (0,T_0)$ for some $T_0\in (0,T]$. Suppose $T_0$ is the first blowup time of $u$, and $Z_0=(T_0,X_0)$ is a singular point. Take a decreasing sequence $\{\lambda_k\}$ converging to $0$. We rescale the pair $(u,p)$ at time $T_0$ and define
\begin{align*}
u_k(t,x)&=\lambda_ku(T_0+\lambda_k^2t,X_0+\lambda_k x),\\
p_k(t,x)&=\lambda_k^2 p(T_0+\lambda_k^2t,X_0+\lambda_k x).
\end{align*}
Then for each $k=1,2,\cdots$, $(u_k,p_k)$ is a suitable weak solution of \eqref{NSE} and $u_k$ is smooth for $t\in (-\lambda_k^{-2}T_0,0)$.

We finish this section by constructing a limiting solution.
\begin{prop}
 					\label{prop3.4}
i) There is a subsequence of $\{(u_k,p_k)\}$, which is still denoted by $\{(u_k,p_k)\}$, such that
\begin{align}
                        \label{eq8.50}
u_k&\to u_\infty\,\, \text{in}\,\, C([t_0-1,t_0];L_{q_1}(B(x_0,1))),\\                        p_k&\rightharpoonup p_\infty\,\, \text{in}\,\, L_{q_2}^tL_{d/2}^x(Q(z_0,1)).\label{eq8.53}
\end{align}
for any $z_0\in (-\infty,0]\times \bR^d$, $q_1\in [1,d)$ and $q_2\in [1,\infty)$.

ii) Furthermore, $(u_\infty,p_\infty)$ is a suitable weak solution of \eqref{NSE} in $(-\infty,0)\times \bR^d$, and
$$
u_\infty\in L_{q_2}^tL_{d}^x((-T_1,0)\times\bR^d),
\quad p_\infty\in L_{q_2}^tL_{d/2}^x((-T_1,0)\times\bR^d).
$$
for any $T_1>0$ and $q_2\in [1,\infty)$.
\end{prop}
\begin{proof}
First we fix a $z_0\in (-\infty,0]\times \bR^d$. Since $p_k,k=1,2,\cdots$ have a uniform bound of the $L_{\infty}^tL_{d/2}^x((t_0-1,t_0)\times\bR^d)$ norm, and consequently  a uniform bound of their $L_{q_2}^tL_{d/2}^x(Q(z_0,1))$ norms, there is a subsequence, which is still denoted by $\{p_k\}$, such that \eqref{eq8.53} holds. Similarly,
\begin{equation}
                            \label{eq9.17}
\|u_k\|_{L_\infty^t L_d^x(Q(z_0,3))}\le
\|u_k\|_{L_\infty^t L_d^x((t_0-9,t_0)\times\bR^d)}\le N,
\end{equation}
where $N$ is independent of $k$. By Lemma \ref{lem2.1}, we have
$$
A(2,z_0,u_k,p_k)+B(2,z_0,u_k,p_k)\le N.
$$
Now following the proof of Lemma \ref{lem3.3}, we deduce
\begin{equation*}
u_k\in L_4(Q(z_0,3/2)),\quad \partial_t u_k,D^2u_k,\nabla p_k\in L_{4/3}(Q(z_0,3/2))
\end{equation*}
with uniform norms. Therefore, we can find a subsequence still denoted by $\{u_k\}$ such that
$$
u_k\to u_\infty\,\, \text{in}\,\, C([t_0-1,t_0];L_{4/3}(B(x_0,1))).
$$
This together with \eqref{eq9.17} gives \eqref{eq8.50} by using the H\"older's inequality. To finish the proof of Part  i), it suffices to use a Cauchy diagonal argument. Part ii) then follows from  Part i) and \eqref{eq9.17}.
\end{proof}

\mysection{Schoen's trick}		\label{sec4}

The objective of this section is to establish the following regularity criterion.

\begin{theorem}
                            \label{thm4.1}
Suppose $u$ is a regular solution of \eqref{NSE} in $Q(z_0,\rho_1)$. Then for any $K>0$ there exists an $\epsilon_1=\epsilon_1(d,K)>0$ such that following is true. If any  $z_1\in Q(z_0,\rho_1/2),\rho\in (0,\rho_1/2)$ we
have
\begin{equation}
                                    \label{eq18.10.21}
C(\rho,z_1)\leq \epsilon_1,\quad
\|p\|_{L^t_\infty L^x_{d/2}(Q(z_1,\rho))}\leq K
\end{equation}
then
\begin{equation*}
\sup_{Q(z_0,\rho_1/4)}|u(z)|<N(\rho_1,d).
\end{equation*}
\end{theorem}
\begin{proof}
We prove the theorem by using the Schoen's trick. Let $\delta\in (0,\rho_1^2/4)$ be a number and denote
$$
d(z)=(t_0+\rho_1^2/4-t)^{1/2},\quad
M_{\delta}=\max_{\bar Q(z_0,\rho_1/2)\cap
\{t\le t_0-\delta\}}d(z)|u(z)|.
$$
If for all $\delta\in (0,\rho_1^2/4)$ we have $M_\delta\leq 2$, then
there's nothing to prove. Otherwise, suppose for some $\delta$ and
$z_1\in \bar Q(z_0,\rho_1/2)\cap \{t\le t_0-\delta\}$,
$$
M:=M_\delta=|u(z_1)|d(z_1)>2.
$$
Let $r_1=d(z_1)/M<d(z_1)/2$.
We make the scaling as follows:
$$
\bar u(y,s)=r_1u(r_1^2s+t_1,r_1y+x_1),
$$
$$
\bar p(y,s)=r_1 p(r_1^2 s+t_1,r_1 y+x_1).
$$
The pair $(\bar u,\bar p)$ satisfies \eqref{NSE} in $Q(0,1)$ and $\bar u$ is smooth. Obviously,
\begin{equation}
                            \label{eq11.30.b}
\sup_{Q(0,1)}|\bar u|\leq 2,\quad |\bar u(0,0)|=1.
\end{equation}

By the scaling-invariant property of the quantity $C$, in what follows we view it as the object associated to $(\bar
u,\bar p)$ at the origin. For any $\rho\in (0,1]$, from \eqref{eq18.10.21} we have
\begin{align}
                                    \label{eq19.10.55}
&C(\rho)\leq \epsilon_1,\\
					\label{eq10.52}
&\|\bar p\|_{L^t_\infty L_{d/2}^x(Q(1))}\le K.
\end{align}
We decompose $\bar p$ as in the proof of Lemma \ref{lem2.3}:
$$
\bar p= {\bar p}_{0,1}+\bar h_{0,1}.
$$
Because of \eqref{eq11.30.b}, we have
\begin{equation}
                            \label{eq19.11.13}
\int_{Q(0,1)}| {\bar p}_{0,1}|^{4(d+2)}\,dz\leq N.
\end{equation}
Since $\bar h_{0,1}(t,\cdot)$ is harmonic in $B(2/3)$ for a.e. $t\in (-1,0)$, it holds that
\begin{align}
&\int_{Q(0,1/2)}| \bar h_{0,1}|^{4(d+2)}\,dz\nonumber\\
&\le N\int_{-1/4}^0\sup_{B(1/2)}| \bar h_{0,1}(t,\cdot)|^{4(d+2)}\,dt.\nonumber\\
&\le N\int_{-1/4}^0 \left(\int_{B(2/3)}| \bar h_{0,1}|^{d/2}\,dx\right)^{8(d+2)/d}\,dt\nonumber\\
&\le N\int_{Q(0,1)}| {\bar p}_{0,1}|^{4(d+2)}\,dz+\sup_{t\in (0,1)}\left(\int_{B(0,1)}| \bar p(t,\cdot)|^{d/2}\,dx\right)^{8(d+2)/d}\nonumber\\
                                    \label{eq12.03}
&\le N,
\end{align}
where in the last inequality we used \eqref{eq10.52} and \eqref{eq19.11.13}. Thus, we deduce from \eqref{eq19.11.13} and \eqref{eq12.03} that
\begin{equation}
                        \label{eq11.26}
\int_{Q(0,1/2)}| \bar p|^{4(d+2)}\,dz\leq N.
\end{equation}

Now we note that $(\bar u,\bar p)$ satisfies the equation
$$
\partial_t \bar u-\Delta \bar u=\text{div}(\bar u\otimes \bar u)-\nabla (\bar p)
$$
in $Q(0,1)$. Owing to \eqref{eq11.30.b}, \eqref{eq11.26} and the
classical Sobolev space theory of parabolic equations, we have
\begin{equation*}
\bar u\in W_{4(d+2)}^{1,1/2}(Q(0,1/3)),\quad \|\bar
u\|_{W_{4(d+2)}^{1,1/2}(Q(0,1/3))}\leq N.
\end{equation*}
By the Sobolev embedding theorem (see
\cite{O.A.}), we obtain
$$
\bar u\in C^{1/4}(Q(0,1/4)),\quad \|\bar
u\|_{C^{1/4}(Q(0,1/4))}\leq N,
$$
where $N$ is a universal constant depending only on $d$ and $K$. Therefore, we can find $\delta_1<1/5$ independent of
$\epsilon_1$ such that
\begin{equation}
                                \label{eq19.12.23}
|\bar u(x,t)|\geq 1/2 \quad\text{in}\,\,Q(0,\delta_1).
\end{equation}
Now we choose $\epsilon_1$ small enough which makes
\eqref{eq19.12.23} and \eqref{eq19.10.55} a contradiction. The theorem is proved.
\end{proof}

\mysection{Proof of Theorem \ref{thm1} and \ref{thm2}}			 \label{sec5}

We finish the proof of Theorem \ref{thm1} in this section.
Let $u_k$, $p_k$, $u_\infty$ and $p_\infty$ be the functions constructed in Section \ref{sec3}.
First we verify that the assumptions of Theorem \ref{thm4.1} hold for $(u_k,p_k)$ when $k$ is sufficiently large and the parabolic cylinder is far away from the origin.

\begin{lemma}
                        \label{lem5.1}
For any $\epsilon_2>0$ and $T_1\ge 1$, we can find $R\ge 1$ such that, for any $z_0\in (-T_1-1,0]\times(\bR^d\setminus B_{R+1})$,
\begin{equation}
                    \label{eq9.55}
\limsup_{k\to \infty} C(1,z_0,u_k,p_\infty)\le \epsilon_2.
\end{equation}
\end{lemma}
\begin{proof}
Due to Proposition \ref{prop3.4} ii), we can find $R$  large such that
$$
\int_{(-T_1-2,0)\times(\bR^d\setminus B_R)}|u_\infty|^d\,dz
$$
is sufficiently small. This together with Proposition \ref{prop3.4} i) proves the lemma.
\end{proof}

\begin{lemma}
 			\label{lem5.2}
For any $\epsilon_3>0$ and $T_1\ge 1$, we can find $R\ge 1$ and $\rho_3\in (0,1/2]$ such that, for any $\rho\in (0,\rho_3]$ and $z_0\in (-T_1-1,0]\times(\bR^d\setminus B_{R+2})$,
\begin{equation}
                    \label{eq18.10.29}
\limsup_{k\to \infty}\left(C(\rho,z_0,u_k,p_k)+D(\rho,z_0,u_k,p_k)\right)
\le \epsilon_3.
\end{equation}
\end{lemma}
\begin{proof}
The lemma is a consequence of Lemma \ref{lem5.1}, \ref{lem2.3} and Proposition \ref{prop3.1}. Indeed, since $D(1,z_0,u_k,p_k)$ has a uniform bound, for any $\epsilon>0$, we can choose $\gamma$ small in \eqref{eq11.29}, then $\epsilon_2$ small in \eqref{eq9.55} and $R$ large such that
\begin{equation*}
\limsup_{k\to \infty}\left(C(\gamma,z_0,u_k,p_k)+D(\gamma,z_0,u_k,p_k)\right)
\le \epsilon
\end{equation*}
holds for any $z_0\in (-T_1-1,0]\times(\bR^d\setminus B_{R+1})$. Now it suffices to choose $\epsilon$ small depending on $\epsilon_3$ and apply Proposition \ref{prop3.1}. We finish the proof by setting $\rho_3=\gamma/2$.
\end{proof}

Next we show that $u_\infty$ is identically equal to zero.

\begin{prop}
				\label{prop5.3}
Under the assumptions of Theorem \ref{thm1}, let $(u_\infty,p_\infty)$ be the suitable weak solution constructed in Section \ref{sec3}. Then,
$$
u_\infty(t,\cdot)\equiv 0\quad \forall\,\,t\in(-\infty,0).
$$
\end{prop}
\begin{proof}
Let $\epsilon_1$ be the constant in Theorem \ref{thm4.1}. Let $T_1\ge 1$ be a number. Owing to lemma \ref{lem5.2}, we can find $R\ge 1$ and $\rho_3\in (0,1/2]$ such that, for any $\rho\in (0,\rho_3]$ and $z_0\in (-T_1-1,0]\times(\bR^d\setminus B_{R+2})$ estimate \eqref{eq18.10.29} holds with $\epsilon_1/2$ in place of $\epsilon_3$. Moreover, we recall that for each $K$
$$
\|p_k\|_{L^t_\infty L^x_{d/2}((-\infty,0)\times \bR^d)}\leq N(d)K.
$$
Thus Theorem \ref{thm4.1} yields that
$$
\limsup_{k\to \infty}\sup_{Q(z_0,\rho_3/4)}|u_k(z_0)|\le N(d,\rho_3)
$$
for any $z_0\in [-T_1-1,0)\times(\bR^d\setminus B_{R+2})$.
Now by Proposition \ref{prop3.4}, we obtain
$$
|u_\infty(z)|\le N(d,\rho_3)
$$
for a.e. $z\in [-T_1-1,0)\times(\bR^d\setminus B_{R+2})$. Upon using the regularity results for linear Stokes systems, one can estimate higher derivatives
\begin{equation}
 				\label{eq4.44}
|D^j u_\infty(z)|\le N(d,j,\rho_3)
\end{equation}
for any $j\ge 1$ and a.e. $z\in [-T_1,0)\times(\bR^d\setminus B_{R+3})$.

We now claim that $u_\infty(0,\cdot)\equiv 0$ by adapting the argument in the proof of Theorem 1.4 \cite{ESS}. For any $x_0\in \bR^d$, by using \eqref{eq8.50},
\begin{align*}
 &\int_{B(x_0,1)}|u_\infty(x,0)|\,dx\\
&\le \int_{B(x_0,1)}|u_k(x,0)-u_\infty(x,0)|\,dx
+\int_{B(x_0,1)}|u_k(x,0)|\,dx\\
&\le \|u_k-u\|_{C([-1,1];L_1(B(x_0,1)))}
+N(d)\left(\int_{B(x_0,1)}|u_k(x,0)|^d\,dx\right)^{1/d}\\
&=\|u_k-u\|_{C([-1,1];L_1(B(x_0,1)))}
+N(d)\left(\int_{B(\lambda_k x_0,\lambda_k)}|u(y,0)|^d\,dy\right)^{1/d}.
\end{align*}
The right-hand side of the above inequality goes to zero as $k\to \infty$, which proves the claim.

Because of \eqref{eq4.44}, the vorticity $\omega=\text{curl}\, u_\infty$ satisfies the differential inequality
\begin{equation*}
|\partial_t\omega-\Delta\omega|
\le N(|\omega|+|\nabla\omega|)
\end{equation*}
on $(-T_1,0]\times(\bR^d\setminus B_{R+3})$. Thanks to the backward uniqueness theorem proved in \cite{ESS} (see also \cite{ESS2}), we reach
\begin{equation}
 				\label{eq18.4.49}
\omega(z)=0\quad\text{on}\,\,(-T_1,0]\times(\bR^d\setminus B_{R+3}).
\end{equation}

Now we fix a $t_0\in (-T_1,0)$. Take a increasing sequence $\{t_k\}\subset (-T_1,0)$ converging to $t_0$. For each $k$, we consider equation \eqref{NSE} with initial data $u_\infty(t_k,\cdot)$. By Proposition \ref{prop3.53}, one can locally find a strong solution
$$
v_k\in C([t_k,t_k+\delta_k);L_d(\bR^d)).
$$
for some small $\delta_k$, and $v_k(t,\cdot)$ is spatial analytic for $t\in (t_k,t_k+\delta_k)$. By the weak-strong uniqueness, $v_k\equiv u_\infty$ for $t\in [t_k,t_k+\delta_k)$. Therefore, $\omega(t,\cdot)$ is also spatial analytic for $t\in (t_k,t_k+\delta_k)$. Because of \eqref{eq18.4.49}, we get
\begin{equation*}
\omega(z)=0\quad\text{on}\,\,(t_k,t_k+\delta_k)\times\bR^d,
\end{equation*}
which implies that $u_\infty\equiv 0$ in the same region. In particular, there exists a sequence $\{s_k\}$ converging to $t_0$ such that
$$
t_k< s_k\le t_0,\quad u_\infty(s_k,\cdot)\equiv 0.
$$
This together with the weak continuity of $u_\infty$ yields that $u_\infty(t_0,\cdot)\equiv 0$. Since $t_0\in (-T_1,0)$ is arbitrary and $T_1\ge 1$ is also arbitrary, we complete the proof of the theorem.
\end{proof}

We are ready to prove Theorem \ref{thm1}.

\begin{proof}[Proof of Theorem \ref{thm1}] We prove the theorem in three steps.

{\em Step 1.} First we show that $u$ is regular for $t\in (0,T]$. Thanks to Proposition \ref{prop3.4} and \ref{prop5.3},
$$
u_k\to 0\,\,\text{in}\,\, C([-3,0];L_{2(d+3)/(d+1)}(B(3))).
$$
Also recall that $D(1,z_0,u_k,p_k)$ has a uniform bound.
Following the proof of Lemma \ref{lem5.2} we have: for any $\epsilon_4>0$, there is a $\rho_4\in (0,1/2]$ and a positive integer $k_0$ such that, for any $\rho\in (0,\rho_3]$ and $z_0\in (-2,0]\times B(2)$,
\begin{equation*}
C(\rho,z_0,u_{k_0},p_{k_0})+D(\rho,z_0,u_{k_0},p_{k_0})
\le \epsilon_4.
\end{equation*}
We choose $\epsilon_4$ sufficiently small and apply Theorem \ref{thm4.1} to get
\begin{equation*}
\sup_{(-1,0)\times B(1)}|u_{k_0}|<\infty,
\end{equation*}
which implies that
\begin{equation*}
\sup_{Q(Z_0,\lambda_{k_0})}|u|<\infty.
\end{equation*}
This contradicts the assumption that $(T_0,X_0)$ is a blowup point. Therefore, $u$ is regular for $t\in (0,T]$.

{\em Step 2.} We bound the sup norm of $u$ in this step. Fix a $\delta\in (0,T)$. Since
$$
\|u\|_{L^t_\infty L^x_{d}((0,T)\times \bR^d)}\leq N,\quad
\|p\|_{L^t_\infty L^x_{d/2}((0,T)\times \bR^d)}\leq N,
$$
by the same reasoning as at the beginning of the proof of Proposition \ref{prop5.3}, we see that there exists a large $R\ge 1$ such that
\begin{equation}
                                \label{eq23.4.27}
\sup_{[\delta,T)\times (\bR^d\setminus B(R))}|u|\le N.
\end{equation}
Next we estimate the sup norm of $u$ in $[\delta,T)\times B(R)$. Fix a $z_0=(t_0,x_0)$ in $[\delta,T]\times \bar B(R)$. In the construction of $u_k$, we replace $(T_0,X_0)$ by $(t_0,x_0)$. By the same reasoning as in the first step, for some $\epsilon=\epsilon(T_0,X_0)>0$, we have
\begin{equation*}
\sup_{Q(z_0,\epsilon)}|u|<\infty.
\end{equation*}
By the compactness of $[\delta,T]\times \bar B(R)$, it holds that
\begin{equation*}
\sup_{[\delta,T)\times \bar B(R)}|u|\le N.
\end{equation*}
This together with \eqref{eq23.4.27} yields
\begin{equation*}
\sup_{[\delta,T)\times \bR^d}|u|\le N.
\end{equation*}

{\em Step 3.} Finally we prove the uniqueness. Owing to the local strong solvability of \eqref{NSE}, we have $u\in L_{d+2}(\bR^{d+1}_{T_1})$ for some $T_1\in (0,T)$. On the other hand, for $t\in [T_1,T]$ the solution is uniformly bounded and belongs to $L_\infty^tL_d^x((T_1,T)\times \bR^d)$, thus $u\in L_{d+2}(\bR^{d+1}_{T})$. The uniqueness then follows.
\end{proof}

Now we give
\begin{proof}[Proof of Theorem \ref{thm2}] Thanks to Theorem \ref{thm1}, it remains to prove \eqref{eq10.22}. 
Let $\lambda>0$ be a constant to be specified later. We define
\begin{align*}
u_\lambda(t,x)&=\lambda u(\lambda^2t,\lambda x),\\
p_\lambda(t,x)&=\lambda^2 p(\lambda^2t,\lambda x).
\end{align*}
Then $(u_\lambda,p_\lambda)$ is also a Leray-Hopf weak solution of \eqref{NSE} in $(0,\infty)\times\bR^d$, and $u_\lambda$ satisfies \eqref{eq11.11b} with the same constant $K$ due to the scaling invariant property.

By the proof of Lemma \ref{lem3.3}, we have $u_\lambda\in L_4((0,\infty)\times \bR^d)$. Thus for any $\epsilon>0$, there is a $T>0$ such that $\|u_\lambda\|_{L_4((T,\infty)\times \bR^d)}\le \epsilon$. Let $\epsilon_1$ be the constant in Theorem \ref{thm4.1}. Upon using Lemma \ref{lem2.3} and Proposition \ref{prop3.1}, we can find a large $T=T_\lambda$ such that
$$
C(\rho,z_0,u_\lambda,p_\lambda)+D(\rho,z_0,u_\lambda,p_\lambda)\le \epsilon_1,
$$
for any $\rho\in (0,1/2]$ and $z_0\in [T,\infty)\times \bR^d$. Owning to Theorem \ref{thm4.1}, we conclude
$$
\sup_{Q(z_0,1/4)}|u_\lambda(z)|<N,
$$
where $N=N(d,K)$ is independent of $\lambda$. Therefore,
$$
\sup_{t\ge\lambda^2 T,x\in \bR^d}|u(t,x)|<N/\lambda.
$$
Sending $\lambda\to \infty$ yields the desired result. The theorem is proved.
\end{proof}

\section*{Acknowledgment}
The authors would like to express their sincere gratitude to Vladmir \v Sver\'ak for very helpful comments and suggestions. The authors are also grateful to Gabriel Koch and the referee for useful comments on a previous version of the manuscript.


\begin{thebibliography}{m}
\bibitem{CKN} L. Caffarelli, R. Kohn, and L. Nirenberg, Partial
regularity of suitable weak solutions of the Navier-stokes
equations, \textit{Comm. Pure Appl. Math.} \textbf{35} (1982),
771--831.

\bibitem{CS} A. Cheskidov and R. Shvydkoy,
On the regularity of weak solutions of the 3D Navier-Stokes
equations in $B^{-1}_{\infty,\infty}$, \textit{Preprint}
arXiv:math.AP/0708.3067 (2007).


\bibitem{DongDu07} H. Dong, D. Du, Partial regularity of
solutions to the four-dimensional Navier-Stokes equations at the first blow-up time, \textit{Comm. Math. Phys.} \textbf{273} (2007), no. 3, 785--801.

\bibitem{DongLi} H. Dong, D. Li, Optimal local smoothing and analyticity  rate estimates for the generalized Navier-Stokes equations, \textit{Comm. Math. Sci.}, to appear (2008).

\bibitem{ESS} L. Escauriaza, G. Seregin, V. \v Sver\'ak,
$L_{3,\infty}$-solutions of Navier-Stokes equations and backward uniqueness, (Russian) \textit{Uspekhi Mat. Nauk} \textbf{58} (2003), no. 2(350), 3--44; translation in \textit{Russian Math. Surveys} \textbf{58} (2003), no. 2, 211--250.

\bibitem{ESS2} L. Escauriaza, G. Seregin, V. \v Sver\'ak, Backward uniqueness for parabolic equations, \textit{Arch. Ration. Mech. Anal.} \textbf{169} (2003), no. 2, 147--157.

\bibitem{GIP} I. Gallagher, D. Iftimie, F. Planchon, Asymptotics and stability for global solutions to the Navier-Stokes equations,  \textit{Ann. Inst. Fourier (Grenoble)}  \textbf{53}  (2003),  no. 5, 1387--1424.


\bibitem{giga86} Y. Giga, Solutions for semilinear parabolic
equations in $L^p$ and regularity of weak
solutions of the Navier-Stokes system, \textit{J. Differential Eq.}, \textbf{62} (1986), 186--212.


\bibitem{giga2} Y. Giga, T. Miyakawa, Solution in $L_r$ of the Navier-Stokes initial value problem. \textit{Arch. Rational Mech. Anal.,} \textbf{89} (1985), 267--281.

\bibitem{GigaSawada}  Y. Giga, O. Sawada, On regularizing-decay
rate estimates for solutions to the Navier-Stokes initial value problem,
\textit{Nonlinear analysis and applications: to V. Lakshmikantham on
his 80th birthday.} \textbf{1,2}, 549--562, Kluwer Acad. Publ., Dordrecht, 2003.

\bibitem{tsai} S. Gustafson, K. Kang,
T. Tsai, Interior regularity criteria for suitable weak solutions of
the Navier-Stokes equations, \textit{Comm. Math. Phys.} \textbf{273}  (2007),  no. 1, 161--176.


\bibitem{Hopf} E. Hopf, \"Uber die Anfangswertaufgabe f\"ur
die hydrodynamischen Grundgleichungen, \textit{Math. Nachr.} \textbf{4} (1951), 213--231.


\bibitem{kato} T. Kato, Strong $L^p$-solutions of the Navier-Stokes equation in
${\mathbb R}^m$ with applications to weak solutions, \textit{Math.
Z.} \textbf{187} (1984), 471--480.

\bibitem{tataru} H. Koch, D. Tataru, Well-posedness for the Navier-Stokes equations, \textit{Adv. Math.} \textbf{157} (2001), no. 1, 22--35.

\bibitem{Lady} O. Ladyzhenskaya, On the uniqueness and
smoothness of generalized solutions to the Navier-Stokes equations, \textit{Zap. Nauchn. Sem. Leningrad. Otdel. Mat. Inst. Steklov.
(LOMI)} \textbf{5} (1967), 169--185; English transl., \textit{Sem. Math. V. A. Steklov Math. Inst.
Leningrad 5} (1969), 60--66.

\bibitem{O.A.} O. Ladyzhenskaya, V. Solonnikov, N. Ural'tseva, \textit{Linear and quasi-Linear equations of parabolic type}, Nauka, Moscow, 1967 (in Russian); English
translation: Amer. Math. Soc., Providence, RI, 1968.

\bibitem{OL} O. Ladyzhenskaya, \textit{The Mathematical Theory of Viscous Incompressible Flows (2nd edition)}, Gordon and Breach, 1969.

\bibitem{OL2} O. Ladyzhenskaya and G. A. Seregin, On partial regularity of suitable weak solutions to the three-dimensional Navier--Stokes
equations \textit{J. Math. Fluid Mech.} \textbf{1} (1999).

\bibitem{MiShi} A. Mikhailov, T. Shilkin,
$L_{3,\infty}$-solutions to the 3D-Navier-Stokes system in the domain with a curved boundary, (English, Russian summary) \textit{Zap. Nauchn. Sem. S.-Peterburg. Otdel. Mat. Inst. Steklov. (POMI)} \textbf{336} (2006), Kraev. Zadachi Mat. Fiz. i Smezh. Vopr. Teor. Funkts. 37, 133--152, 276; translation in
\textit{J. Math. Sci. (N. Y.)} \textbf{143} (2007), no. 2, 2924--2935.

\bibitem{leray} J. Leray, \'Etude de diverses \'equations int\'egrales non
lin\'eaires et de quelques probl\`emes que pose l'hydrodynamique,
\textit{J. Math. Pures Appl.} \textbf{12} (1933), 1--82.

\bibitem{flin} F. Lin, A new proof of the Caffarelli-Kohn-Nirenberg
theorem, \textit{Comm. Pure Appl. Math.} \textbf{51} (1998),
241--257.




\bibitem{Prodi} G. Prodi, Un teorema di unicit\`a per le
equazioni di Navier-Stokes, \textit{Ann. Mat. Pura
Appl.} \textbf{48} (1959), 173--182.

\bibitem{Sch1} V. Scheffer, Partial regularity of solutions to the
Navier-Stokes equations, \textit{pacific J. Math.} \textbf{66}
(1976), 535--552.

\bibitem{Sch2} V. Scheffer, Hausdorff measure and the Navier-Stokes
equations, \textit{Comm. Math. Phys.} \textbf{55} (1977), 97--112.


\bibitem{Sch4} V. Scheffer, The Navier-Stokes equations on a bounded
domain, \textit{Comm. Math. Phys.} \textbf{73} (1980), 1--42.

\bibitem{Seregin05} G. Seregin,
On smoothness of $L_{3,\infty}$-solutions to the Navier-Stokes equations up to boundary, \textit{Math. Ann.} \textbf{332} (2005), no. 1, 219--238.

\bibitem{SS} G. Seregin, V. Sverak,  On smoothness of
suitable weak solutions to the Navier-Stokes equations, \textit{Zap. Nauchn. Sem. S.-Peterburg. Otdel. Mat. Inst. Steklov. (POMI)} \textit{306} (2003), Kraev. Zadachi Mat. Fiz. i Smezh. Vopr. Teor. Funktsii. 34, 186--198, 231; translation in \textit{J. Math. Sci. (N. Y.)} \textbf{130} (2005), no. 4, 4884--4892.

\bibitem{Serrin62} J. Serrin, On the interior regularity of
weak solutions of Navier-Stokes equations, \textit{Arch. Rat. Mech.
Anal.} \textbf{9} (1962), 187--195.

\bibitem{MareSolo} P. Maremonti, V. Solonnikov,
On estimates for the solutions of the nonstationary Stokes problem in S. L. Sobolev anisotropic spaces with a mixed norm. (Russian. English, Russian summary) \textit{Zap. Nauchn. Sem. S.-Peterburg. Otdel. Mat. Inst. Steklov. (POMI)} 222 (1995), Issled. po Linein. Oper. i Teor. Funktsii. 23, 124--150, 309; \textit{translation in
J. Math. Sci. (New York)} 87 (1997), no. 5, 3859--3877

\bibitem{Serrin63} J. Serrin, The initial value problem for
the Naiver-Stokes equations, \textit{Nonlinear problems} (R. Langer, ed.), Univ. of Wisconsin Press, Madison 1963, 69--98.


\bibitem{Struwe} M. Struwe, On partial regularity results for the Navier-Stokes equations, \textit{Comm. Pure Appl. Math.} \textbf{41} (1988), 437--458.

\bibitem{Wahl} W. von Wahl, The Equations of Navier-Stokes and Abstract Parabolic Equations, Vieweg, Braunschweig,
1985.

\bibitem{Taylor} M. Taylor, Analysis on Morrey spaces and
applications to Navier-Stokes equation, \textit{Comm. Partial Differential Equations}, \textbf{17} (1992), 1407--1456.


\end{thebibliography}
\end{document}